\documentclass[11pt]{article}
\usepackage{amsmath,amssymb,amsthm,mathrsfs}
\usepackage{hyperref}

\newtheorem{theorem}{Theorem}
\newtheorem{lemma}{Lemma}
\newtheorem{proposition}{Proposition}

\newtheorem{remark}{Remark}
\newcommand{\cH}{\mathcal{H}}
\newcommand{\cG}{\mathcal{G}}
\newcommand{\cF}{\mathcal{F}}
\newcommand{\cU}{\mathcal{U}}
\newcommand{\cV}{\mathcal{V}}
\newcommand{\cC}{\mathcal{C}}
\newcommand{\cS}{\mathcal{S}}

\title{Cyclic AG-Codes on the Hermitian Curve}
\author{Angela Aguglia\footnote{Angela Aguglia: angela.aguglia@poliba.it %\hfill\newline
Dipartimento di Meccanica, Matematica e Management,- Politecnico di Bari- Via E. Orabona, 4 - 70125 Bari (Italy).}%\hspace*{1.4em}
\newline \and
G\'abor Korchm\'aros, \footnote{G\'abor Korchm\'aros: gabor.korchmaros@unibas.it Dipartimento di Scienze di Base e Applicazioni, - Universit\`{a} degli Studi della Basilicata - Viale dell'Ateneo Lucano 10 - 85100 Potenza (Italy)}
}
\date{}
\begin{document}
\maketitle

\begin{abstract}
Cyclic AG-codes $C_\mathcal{L}({\mathtt{D}},{\mathtt{G}})$ on the Hermitian curve $\cH_q$ over $\mathbb{F}_{q^2}$ are constructed such that $\mathtt{G}=m(P_2+\ldots + P_{q})$ where  $2\le m \le q-1$ and ${\rm{supp}}(\mathtt{G})$ is the intersection of $\cH_q$ with a chord $\ell$ of $\cH_q$ minus two points $P_1,P_{q+1}$, while  $\mathtt{D}=Q_1+\ldots +Q_{q^2-1}$ where ${\rm{supp}}(\mathtt{D})$ comprises all the $q^2-1$ points of a single orbit under the action of the (cyclic) $2$-point stabilizer $\Gamma$ of $(P_1,P_{q+1})$ in ${\rm{Aut}}(\cH_q)=\mathrm{PGU}(3,q)$. 
\end{abstract}

\section{Introduction} A \emph{cyclic code} $C$ is a linear code of length $n$ such that the set of the codewords is invariant under an $n$-cycle permutation on the coordinates, that is, for any codeword $c=(c_1,c_2,...,c_n)$ in $C$, the cyclic shift $s(c) = s(c_1,c_2,...,c_n) = (c_2,...,c_n,c_1)$ is also a codeword in $C$.

Cyclic codes are among the most important linear codes since they have good parameters, excellent detection-correction capabilities and fast and efficient encoding-decoding algorithms. Also, cyclic codes appear in several important families of codes such as Golay codes, binary Hamming codes, Reed-Solomon codes and BCH codes. 

The systematic study of AG (algebraic-geometry) codes  which are cyclic codes has begun in the recent paper \cite{CCPT} where the procedure for the construction of any functional cyclic AG code $C_\mathcal{L}({\mathtt{D}},{\mathtt{G}})$
is described as follows. Given a non-singular absolutely irreducible, non-singular curve $\mathcal{X}$ defined over a finite field $\mathbb{F}_n$, choose an $\mathbb{F}_n$-automorphism $\sigma$ of $\mathcal{X}$ together with an $\mathbb{F}_n$-rational point $P\in\mathcal{X}$, and then define the support $D$ of $\mathtt{D}$ as the orbit of $P$ under the action of $\sigma$, and take for $\mathtt{G}$  an $\mathbb{F}_n$-rational  $\sigma$-invariant divisor whose support is disjoint from $D$. Then the arising functional AG code $C_{\mathcal{L}}(\mathtt{D},\mathtt{G})$ is cyclic.  The authors have named this procedure \emph{sigma-method}, and their major contribution in \cite{CCPT} concerns $1$-point cyclic AG codes on rational curves $\mathcal{X}$ where $\mathtt{G}=tQ$ for some $\mathbb{F}_n$-rational point $Q\in\mathcal{X}$. 

From previous work on AG codes it has emerged that the best performing linear codes are mostly defined over the Hermitian curve; see \cite{br,KN1,KN2,KNT,KS,XC,YK}   

In this paper, we construct a cyclic functional AG code $C_\mathcal{L}({\mathtt{D}},{\mathtt{G}})$ on the Hermitian curve $\cH_q$ defined over $\mathbb{F}_{q^2}$ such that $\mathtt{G}=m(P_2+\ldots + P_{q})$ where  $2\le m \le q-1$ and ${\rm{supp}}(\mathtt{G})$ is the intersection of $\cH_q$ with a chord $\ell$ of $\cH_q$ minus two points $P_1,P_{q+1}$, while  $\mathtt{D}=Q_1+\ldots +Q_{q^2-1}$ where ${\rm{supp}}(\mathtt{D})$ comprises all the $q^2-1$ points of a single orbit under the action of the (cyclic) $2$-point stabilizer $\Gamma$ of $(P_1,P_{q+1})$ in ${\rm{Aut}}(\cH_q)=\mathrm{PGU}(3,q)$. These cyclic functional AG codes have good performance as their true minimum distance is better (i.e. greater) than the Goppa designed minimum distance with an improvement of at least $q+1-m$; see Theorem \ref{mainth}. Unfortunately, their weight distribution turns out to be heavily dependent on involved intersection patterns of certain plane algebraic curves whose study would go far beyond the scope of the present paper. Nevertheless, the smallest cases $m=2,3$ are thoroughly discussed in Sections \ref{ss2} and \ref{ss3},  the main results being stated in Theorems \ref{th03042026B} and \ref{th03042026}, respectively.

Notation and terminology are standard. Our references are \cite{Hirschfeld1,HKT,hughes-piper1973,P,sti}.

\section{Preliminary Results}\label{sec2}
%Notation and terminology are standard. Our references are \cite{Hirschfeld1,HKT,hughes-piper1973,P,sti}.
%In particular, $\mathbb F_q$ stands for a finite field of order $q$, where $q=p^k$, $p$ prime, $AG(2,q)$ for the affine plane over $\mathbb{F}_q$, and $PG(2,q)$ for the projective closure of $PG(2,q)$.
\subsection{Plane curves and the Riemann-Roch theorem}
For the theory of plane algebraic curves over a field of positive characteristic; see \cite[Chapters 1-5]{HKT}. Fix an algebraic closure $\mathbb{K}$ of a finite field, and let $PG(2, \mathbb{K})$ be the projective plane over $\mathbb{K}$ equipped with homogeneous coordinates $(X_1:X_2:X_3)$. For a non-constant homogeneous polynomial $F(X_1,X_2,X_3)$ over $\mathbb{K}$, the plane curve $\mathcal F$ of equation $F(X_1,X_2,X_3)=0$ is defined to be the set of non-trivial zeros of $F=F(X_1,X_2,X_3)$, i.e.
\[{\bf{v}}(F)=\{(x_1:x_2:x_3)\in PG(2,\mathbb{K})| F(x_1,x_2,x_3) = 0, (x_1,x_2,x_3)\neq (0,0,0)\}.\]
The \textit{degree} of $\mathcal{F}$ is $\deg F$. A \textit{component} of $\mathcal F$ is any curve $\mathcal{G}=\mathbf{v}(G)$ of $PG(2,\mathbb{K})$ such that $G$ divides $F$. A curve $\mathcal F$ is \textit{irreducible} if $F$ is irreducible; otherwise it is \textit{reducible} and splits into irreducible curves, the components of $\mathcal{F}$. Given a positive integer $r$, a point $P\in \cF$ is an \textit{$r$-fold point} when the intersection multiplicity $I(P,\cF\cap \ell)\ge r$ for any line through $P$, and there exists a line $\ell$ for which equality holds. If $\cF$ is irreducible, then there are at most $r$ lines for which  $I(P,\cF\cap \ell)> r$; they are \textit{the tangent lines} to $\cF$ at $P$. If $r\ge 2$, an $r$-fold point is \textit{singular}. For $r=2$,  $P$ is either a \textit{node}, or a \textit{cusp} according as $\cF$ has two or one tangents at $P$.    

For a positive integer $n\ge 1$, take as many as $\frac{1}{2}n(n+3)$ points in $PG(2,\mathbb{K})$. By a classical result, there exists some curve of degree $n$ passing through each of those points.  

From now on $\cF$ is assumed to be smooth, i.e., $\cF$ has no singular point. Thus $\cF$ is irreducible, and its genus $\mathfrak{g}$ equals $\frac{1}{2}(\deg(\cF)-1)(\deg(\cF)-2)$. Let $\mathbb{K}(\cF)$ be the function field of $\cF$ with constant field $\mathbb{K}$.  

The \emph{divisors} are formal sums of points of $\mathcal F$, and for every nonzero function $h$ in $\mathbb{K}(\mathcal F)$,  ${\rm{div}}(h)$ stands for the \textit{principal divisor} associated with $h$.
For a divisor $\mathtt G$ on $\mathcal F$, the \textit{Riemann-Roch space} $\mathcal L(\mathtt G)$ is the vector space consisting of all functions on $\mathcal{F}$ which are regular outside $\mathtt G$ and have no pole at any point $P$ with multiplicity bigger than  the order of $\mathtt{G}$ at $P$, i.e.  $\mathcal L(\mathtt G)=\{f|{\rm{div}} (f)+\mathtt{G}\succcurlyeq 0$\}. The dimension $\ell(\mathtt G)$ of $\mathcal L(\mathtt G)$ and  $\deg(\mathtt G)$ are linked by the \textit{Riemann-Roch Theorem}, see \cite[Theorem 6.70]{HKT}:
\begin{equation}
\label{eq02112025}   \ell(\mathtt G) = \deg(\mathtt G)-\mathfrak  g + 1+\iota(\mathtt{G})
\end{equation}
where $\mathfrak  g$ is the genus of $\mathcal{F}$ and $\iota=\dim(\mathtt{W}-\mathtt{G})$ where $\mathtt W$ is a canonical divisor of $\mathcal{F}$. 
%In particular, for $\deg(\mathtt G)>2\mathfrak  g-2$,
%\[\ell(\mathtt G) = \deg(\mathtt G) -\mathfrak  g+1.\]
%Let $\Omega(\mathtt{G})$ be the vector space of rational differential forms $\omega$ on $\mathcal{F}$ with $\rm{div}(\omega)+\mathtt{G}\succcurlyeq 0$, together with the zero form.
  
Let $\mathcal{U}$ be another plane curve of equation $U(X_1,X_2,X_3)=0$, possibly singular, or reducible, such that $\cF$ is not a component of $\cU$. The intersection divisor $\mathcal{F}\cdot \mathcal{U}$ is defined by $$\sum_{P\in \mathcal{F}\cap\mathcal{U}} I(P,\mathcal{F}\cap\mathcal{U})P$$ 
where $I(P,\mathcal{F}\cap\mathcal{U})$ is the intersection multiplicity of $\mathcal{F}$ and $\mathcal{U}$ in their common point $P$. B\'ezout's theorem, see \cite[Theorem 3.14]{HKT}, states that $\deg(\mathcal{F}\cdot \mathcal{U})= \deg(\mathcal{F})\deg (\mathcal{U})$, that is, 
$$\deg(\mathcal{F})\deg(\mathcal{U})= \sum_{P\in \mathcal{F}\cap\mathcal{U}}I(P,\mathcal{F}\cap \mathcal{U}).$$ 
We will use the classical geometric interpretation of the Riemann-Roch space  based on linear series of curves cut out on $\cF$; see \cite[Chapter 6]{HKT}. Let  
$$ \mathtt{G}=n_1P_1+\ldots+n_sP_s,\quad n_1,\ldots,n_s>0.$$
 Take a curve $\mathcal{U}$ of minimum degree $n$ through ${\mathtt{G}}$, i.e. $\mathcal{F}\cdot \mathcal{U}\succeq \mathtt{G}$ equivalently, $I(P_i,\cF\cap \cU)\ge n_i$ for $i=1,\ldots,s$. Let $r_i=I(P_i,\cF\cap \cU)- n_i$. Suppose $\deg(\cF)>n$. Let $\Lambda$ be the set of all curves $\cV$ of degree $n$ such that $I(P_i,\cF\cap\cV)\ge r_i$. This set is a linear system of finite (projective) dimension, say  $r-1$, and the complete linear series $|\mathtt{G}|$ consists of all divisors cut out on $\mathcal{F}$ by the curves in $\Sigma$, apart from the fixed divisor $\sum_i r_iP_i$. Moreover,  the speciality index $\iota(\mathtt{G})$ is the number of independent canonical divisors passing through the divisor $\mathtt{G}$ where canonical divisors are the intersection divisors between $\cF$ and curves of degree $n-3$. With this setting, the (projective version of the) Riemann-Roch theorem reads
 $$\dim(|\mathtt{G}|)=\deg(\mathtt{G})-\mathfrak{g}+\iota(\mathtt{G}).$$
Take a basis $\{{\bf{v}}(V_1),\ldots,{\bf{v}}(V_{r-1})\}$ of $\Lambda$ where ${\bf{v}}(V_i)=V_i(X_1,X_2,X_3)$ for $i=1,\ldots,r-1$, and let $w_i=V_i(X,Y,1)$, and $u=U(X,Y,1)$. If $\mathbb{K}(\cF)$ is the function field of $\cF$ with generators $x,y$ such that $F(x,y,1)=0$ then $\mathcal{L}(\mathtt{G})$ has dimension $r$ and 
$$\mathcal{L}(\mathtt{G})=\langle 1, w_1(x,y)/u(x,y),\ldots w_{r-1}/u(x,y)\rangle.$$

%The following useful result connecting principal divisors and intersection divisors comes from \cite[Theorem 6.2]{HKT}. Let $\ell_\infty$ denote the line at infinity.  
%\begin{lem}
%\label{theo6.2HKT} Let $(x,y)$ be a generic point of $\mathcal{F}$. If $\mathcal{U}$ has degree $m$ then 
%$$
%{\rm{div}}(U(x,y))=\mathcal{F}\cdot \mathcal{U}-m(\ell_{\infty}\cdot \mathcal{F}).
%$$
%\end{lem}

Now fix a finite subfield $\mathbb{L}$ of $\mathbb{K}$, and assume that the non-singular plane curve $\mathcal{F}$ is defined over $\mathbb{L}$, that is, $\mathcal F$ has equation $F(X_1,X_2,X_3)=0$ with $F(X_1,X_2,X_3)\in \mathbb{L}[X_1,X_2,X_3]$. Let $\mathbb{L}(\cF)$ be the function field of $\cF$ with constant field $\mathbb{L}$.  
Choose a divisor $\mathtt G=\sum \lambda_i Q_i$ where $Q_1,\ldots,Q_k$ are pairwise distinct points defined over  $\mathbb{L}$, and a canonical divisor $\mathtt{W}$ defined over $\mathbb{L}$. Restrict the functions in the Riemann-Roch space of $\mathtt{G}$ to those defined over $\mathbb{L}$ where a function $f(X,Y)$ is defined over $\mathbb{L}$ if $f(X,Y)=g(X,Y)/h(X,Y)$ with $g(X,Y)$,$h(X,Y)\in \mathbb{L}[X,Y]$. By doing so a vector space $\bf{V}$ over $\mathbb{L}$ arises whose dimension remains $\ell(\mathtt{G})$. Also, the Riemann-Roch theorem (\ref{eq02112025}) holds true for $\bf{V}$. Accordingly, we use the term of Riemann-Roch space of $\mathtt{G}$ over $\mathbb{L}$ for $\bf{V}$, and keep the same notation $\mathcal L(\mathtt G)$. In other words, $\mathcal L(\mathtt G)=\{f|{\rm{div}} (f)+\mathtt{G}\succcurlyeq 0, \mbox{$f$ defined over $\mathbb{L}$}\}$. The above geometric interpretation 
holds true over $\mathbb{L}$ whenever $\cV$ is restricted to curves defined over $\mathbb{L}$.  
\subsection{Functional AG codes}
We keep our notation introduced in the previous section. Moreover, let $D$ be a set of points of $\mathcal{F}$ in $PG(2,\mathbb{L})$ other than those in the support $G$ of $\mathtt{G}$. Fix an ordering $(Q_1,Q_2,\ldots,Q_n)$ of the points in $D$, and let $\mathtt{D}=Q_1+\ldots+Q_n$, the associated divisor. 
Let $\mathcal{L}(\mathtt{G})$ be the Riemann Roch space of $\mathtt{G}$ defined over $\mathbb{L}$. 
For any function $f\in\mathcal L(\mathtt G)$ defined over $\mathbb{L}$, the evaluation of $f$ on $\mathtt{D}$ is given by $ev_{\mathtt D}(f) =(f(Q_1),\ldots, f(Q_n))$. 
%Assume that $n>\deg(G)>2\mathfrak{g}-2$. 
The arising evaluation map $ev_{\mathtt D}:\mathcal L(\mathtt G)\rightarrow\mathbb{L}^n$ is $\mathbb{L}$-linear. If  $ev_{\mathtt D}$ is also injective, then its image is the {\emph{functional}} code $C_{\mathcal{L}}(\mathtt{D},\mathtt{G})$ of length $n$, dimension $k=\deg(\mathtt{G})-\mathfrak{g}+1$ and minimum distance $d\geq \delta$ where $\delta=n-\deg(\mathtt{G})$ is the \emph{designed minimum distance}.
$C_{\mathcal{L}}(\mathtt{D},\mathtt{G})$ is \emph{cyclic}, if ${\rm{PGL}}
(3,\mathbb{K})$ has a subgroup $\Gamma$ such that $\Gamma$ preserves $\cF$, $G$, $D$, and acts on $D$ as a cyclic permutation group. 

\subsection{The projective unitary group and the Hermitian curve} 
The projective unitary group $\rm{PGU}(3,q)$ is a subgroup of the projective group $\rm{PGL}(3,q^2)$ of the projective plane $PG(2,q^2)$ defined over the finite field $\mathbb{F}_{q^2}$ of order $q^2$. More precisely, $\rm{PGU}(3,q)$ is the subgroup of $\rm{PGL}(3,q^2)$ which leaves invariant the set of the $q^3+1$ isotropic points of a (non-degenerate) unitary polarity, equivalently the set of all points of a Hermitian curve $\cH_q$ in $PG(2,q^2)$; see \cite{Segre}. For the structure and the action of the subgroups of $\rm{PGU}(3,q)$; see \cite{har,hoffer1972,oli, mi}, and  \cite[Theorem A.10]{HKT}. In particular, the subgroup $\Gamma$ of $\rm{PGU}(3,q)$ fixing two distinct points, say $A_1,A_2$, is a cyclic group of order $q^2-1$ which preserves the set of the other $q-1$ points on the chord $A_1A_2$ of $\cH_q$ and acts semiregularly on the remaining $q^3+1-(q+1)$ points on $\cH_q$. Prior to a suitable change of the projective frame of $PG(2,q^2)$, $\cH_q$ has homogeneous equation $X_2^qX_3+X_2X_3^q-X_1^{q+1}=0$ and $A_1=O, A_2=Y_\infty$ where $O=(0:0:1)$ is the origin and $Y_\infty=(0:1:0)$ is the unique point of $\cH_q$ at infinity. Moreover, the elements of $\Gamma$ are represented by matrices of the form 
\begin{equation}
\label{eq09032026}
M_\lambda=
\left(\begin{array}{ccc}
\lambda & 0 & 0\\
0 & \lambda^{q+1}& 0\\
0 &0& 1
\end{array}\right),
\end{equation}
where $\lambda$ ranges over the non-zero elements of $\mathbb{F}_{q^2}$. Fix a generator $\omega$ of the multiplicative group of $\mathbb{F}_{q^2}$. In $PG(2,q^2)$, if $Q=(u:v:1)$ is a point of $\cH_q$ with $u\neq 0$, then the orbit $\Omega$ of $Q$ consists of all (pairwise distinct) points $Q_i=(\omega^i u:\omega^{(q+1)i}v:1)$ where $i=1,\ldots, q^2-1$,  and $v^q+v=u^{q+1}$. Notice that $u\ne 0$ implies $v^{q-1}+1\ne 0$.
For $\tau\in\mathbb{F}_{q^2}$, $\tau\neq 0$, let $\cC_\tau$ be the rational plane curve of equation $X_2X_3^q-\tau X_1^{q+1}$.
Clearly, $\cC_\tau$ is left invariant under the action of $\Gamma$, and the following claim holds. 
\begin{lemma}
\label{lem09032026A} If $\tau=1/(v^{q-1}+1)$, then the $\Gamma$-orbit $\Omega$ is contained in $\cC_\tau\cap \cH_q$. 
\end{lemma}
\section{Some results on intersection divisors}
Our next step is to determine the intersection divisor between $\cH_q$ and $\cC_\tau$.
%for $\tau=1/(v^{q-1}+1)$.
\begin{lemma}
\label{lem09032026B} If $\tau=1/(v^{q-1}+1)$, then   
\begin{equation}
\label{eq09032026C}
  \cH_q\cdot \cC_{\tau}= (q+1)(O+Y_\infty)+Q_1+Q_2+\ldots+Q_{q^2-1}. 
\end{equation}
\end{lemma}
\begin{proof} We begin by showing that $I(O,\cH_q\cap \cC_\tau)=q+1$. We pass to non-homogeneous coordinates by setting $X=X_1/X_3$ and $Y=X_2/X_3$. Then $\cH_q$ and $\cC_\tau$ have equations $Y^q+Y-X^{q+1}=0$ and $Y=\tau X^{q+1}$, respectively. Replacing $Y$ by $\tau X^{q+1}$ in $Y^q+Y-X^{q+1}$ yields $(\tau-1)X^{q+1}+\tau^qX^{(q+1)q}$. Since $v\neq 0$  (otherwise $u=0$), we have $v^{q-1}+1\neq 1$ and hence  $\tau\ne 1$. Therefore,
$I(O,\cH_q\cap \cC_\tau)=q+1$. 
Next, we show that $I(Y_\infty,\cH_q\cap \cC_\tau)\ge q+1$. Since $Y_\infty$ is a $q$-fold point of $\cC_\tau$ and a simple point of $\cH_q$ and the line at infinity of equation $X_3=0$ is a common tangent to $\cC_\tau$ and $\cH_q$ at $Y_\infty$, the claim follows from \cite[Proposition 3.6]{HKT}. Now, as $Q_i\in \cH_q\cap \cC_\tau$, we also have $I(Q_i,\cH_q\cap \cC_\tau)\ge 1$. Since $\deg(\cH_q\cdot \cC_{\tau})=(q+1)^2$ by B\'ezout's theorem, this yields (\ref{eq09032026C}).   
\end{proof}
Let $\mathcal{S}$ be the completely reducible curve of degree $m\ge 2$ splitting into $m$ lines $\ell_1,\ldots,\ell_m$ through $X_\infty=(1:0:0)$ whose equations are 
\begin{equation}
\label{eq10032026B}
\ell_m: X_3=0,\,\, \ell_{m-1}: X_2=0,\,\, \ell_i: X_2-a_iX_3=0,\, i=1,2,\ldots, m-2
\end{equation}
where $a_i/\tau\in \mathbb{F}_q$ with $a_i\neq 0$. 
\begin{lemma}
\label{lem09032026D}  If $\tau=1/(v^{q-1}+1)$, then   
\begin{equation}
\label{eq09032026CC}
  \cS\cdot \cC_{\tau}= (q+1)(O+Y_\infty)+\sum_{i=1}^{m-2} (Q_{i,1}+Q_{i,2}+\ldots+Q_{i,q+1}). 
\end{equation}
\end{lemma}
\begin{proof} $I(Y_\infty, \cC_\tau\cap \ell_m)\ge q+1$, as $Y_\infty$ is a $q$-fold point of $\cC_\tau$ and $\ell_m$ is the tangent to $\cC_\tau$ at $Y_\infty$. Also,  $I(O,\cC_\tau \cap \ell_{m-1})=q+1$, since $O$ is a simple point of $\cC_\tau$ and $O$ is also the unique common point of $\cC_\tau$ and $\ell_{m-1}$. 
Moreover, since $(a_i/\tau)^{q-1}=1$, $\ell_i$ intersect $\cC_\tau$ in $q+1$ pairwise distinct points, say $Q_{i,1},\ldots, Q_{i,{q+1}}$. Thus $I(Q_{i,j},\cC_\tau\cap \ell_i)=1$ for $1\le i \le m-2$ and $1\le j\le q+1$. Since $\deg(\cS\cdot \cC_{\tau})=m(q+1)$, the claim follows.   
\end{proof} 
Now, let $P_1=O$, $P_i=(0:b_i:1)$ with %$\mathfrak{T}(b)=b^q+b=0$, 
$b_i^q+b_i=0$ and $P_{q+1}=Y_\infty$. Then the chord $OY_\infty$ intersects $\cH_q$ in the points $P_1,P_2,\ldots,P_{q+1}$. 

From now on, $q\ge 3$ is assumed. Fix an integer $2\le m \le q-1$, and let   
\begin{equation}
\label{eq09032026E}
\mathtt{G}=m(P_2+\ldots + P_{q}).
\end{equation}
\begin{proposition}
\label{pro09032026} The linear series $|\mathtt{G}|$ is cut out, apart from the fixed divisor $m(O+Y_\infty)$, by the linear system consisting of all curves $\cV$ of equation 
\begin{equation}
\label{eq10032026A}
X_2X_3g(X_1,X_2,X_3)+\varepsilon X_1^m=0
\end{equation}
where $g(X_1,X_2,X_3)$ runs over all homogeneous polynomials of degree $m-2$. Moreover, $\deg(|\mathtt{G}|)=m(q-1)$ and $\dim(|\mathtt{G}|)=\textstyle{\frac{1}{2}}m(m-1).$
\end{proposition}
\begin{proof} Take the line of equation $X_1=0$ with multiplicity $m$. The arising curve (cycle) $\cU$ of degree $m$ passes through the points in the support of $\mathtt{G}$ and $\cH_q\cdot \cU=\mathtt{G}+m(O+Y_\infty)$. Therefore, the linear system $\Lambda$ consists of all curves $\cV$ of degree $m$ such that $I(O,\cH_q\cap \cV)\ge m$ and  $I(Y_\infty,\cH_q\cap \cV)\ge m$. The linear series $|\mathtt{G}|$ has length $m(q-1)$ and its index of speciality $\iota(\mathtt{G})$ equals $\frac{1}{2}(q-1-m)(q-m)$. In fact, the canonical divisors of $\cH_q$ are cut out by the curves of degree $q-2$ and those passing through  the points in the support of $\mathtt{G}$ have equations $X_1^m H(X_1,X_2,X_3)=0$ with $\deg(H(X_1,X_2,X_3))=q-2-m$. Therefore, $\iota(\mathtt{G})=\frac{1}{2}(q-1-m)(q-m)$. From  the (projective version of the) Riemann-Roch theorem,
$$\dim(|\mathtt{G}|)=m(q-1)-\textstyle{\frac{1}{2}}(q^2-q)+\textstyle{\frac{1}{2}}(q-1-m)(q-m)=\textstyle{\frac{1}{2}}m(m-1).$$
Therefore, $\dim(\Lambda)=\textstyle{\frac{1}{2}}m(m-1)$. On the other hand, the curves of equation (\ref{eq10032026A}) with $\deg(g(X_1,X_2,X_3)=m-2$ belong to $\Lambda$. These curves form a linear system $\Omega$ of dimension $\textstyle{\frac{1}{2}}m(m-1)$. Therefore $\Lambda=\Omega$.  
 \end{proof}
\begin{remark}
\label{rem19032026} \emph{The proof of Proposition \ref{pro09032026} can be used to show that if $m=1$ then  
$\dim(|\mathtt{G)}|)=0$, and $\Lambda$ consists of a unique curve, namely the line of equation $X_1=0$.   This shows that the case $m=1$ is trivial, and it justifies our hypothesis $m\ge 2$.  }
\end{remark}
 
In terms of function field theory, we have the following result. 
\begin{proposition} \label{pro10032026} 
Let $\mathbb{F}_{q^2}(x,y)$ with $y^q+y-x^{q+1}=0$ be the function field of 
 the Hermitian curve $\cH_q$ in its canonical affine equation $Y^q+Y-X^{q+1}=0$. On the chord $OY_\infty$ of $\cH_q$, take the points $Q_1,Q_2,\ldots,Q_{q-1}$ of $\cH_q$ other than $O$ and $Y_\infty$. For an integer $m$ with $2\le m \le q-1$, let $\mathtt{G}=m(Q_1+\ldots+Q_{q-1})$.
Then, 
$$\mathcal{L}(\mathtt{G})=\{f| f=\frac{yg(x,y)+\varepsilon x^m}{x^m}, \deg(g(X,Y))\le m-2 \},$$
and
$$\deg(\mathcal{L}(\mathtt{G}))=m(q-1),\,\,\dim(\mathcal{L}(\mathtt{G}))=\textstyle{\frac{1}{2}}m(m-1)+1.$$ 
\end{proposition}
%\begin{equation}
%\label{eq09032026E}
%\mathtt{G}=n(P_1+P_{q+1})+m(P_2+\ldots + P_{q-1}).
%\end{equation}
\section{A family of cyclic functional AG codes on the Hermitian curve}
We keep up our notation $\cH_q$, $\Gamma$, $\cC_\tau$, $m$, $\mathtt{G}$ from Section \ref{sec2}. Moreover, let $D=\Omega$, that is, for a generator $\omega$ of the multiplicative group of $\mathbb{F}_{q^2}$, $D$  comprises all (pairwise distinct) points $Q_i=(\omega^i u:\omega^{(q+1)i}v:1)$ where $i=1,\ldots, q^2-1$,  and $(u,v)$ with $v^q+v=u^{q+1}$ and $u\ne 0$, represents a fixed affine point of $\cH_q$. 
\begin{theorem}
\label{mainth}
The functional algebraic geometry code $C_\mathcal{L}(\mathtt{D},\mathtt{G})$ is an 
 $$[q^2-1,\textstyle{\frac{1}{2}}m(m-1)+1,d]_{q^2}$$
linear code whose minimum distance $d$ is at most $q^2-1-(m-2)(q+1)$ and at least $q^2-1-(q(m-1)-1)$. The improvement on the designed minimum distance is at least $q+1-m$. 
\end{theorem}
\begin{proof} We begin by showing that $ev_{\mathtt D}$ is injective. If there exists some non-zero function  $f\in \mathcal{L}(\mathtt{G})$ such that $f(Q_i)=0$ for any $1\le i\le q^2-1$, then Proposition \ref{pro09032026} ensures the existence of a curve $\cV$ of equation (\ref{eq10032026A}) passing through each $Q_i$. If this is the case, then there exists a curve $\mathcal{W}$ of degree $m-2$ with the same property, so that each point $Q_i$ is shared by $\cC_\tau$ and $\mathcal{W}$. Since $\cC_\tau$ is irreducible over $\mathbb{K}$, and $\deg(\cC_\tau)=q+1>m-2$, B\'ezout's theorem yields that the number of common points of $\cC_\tau$ and $\mathcal{W}$ does not exceed $(q+1)(m-2)$. Actually, this number is smaller than $q^2-1$, and therefore there is no curve $\cV$ with the required property, and hence if $f\in\mathcal{L}(\mathtt{G})$ and $f(Q_i)=0$ for $i=1,\ldots,q^2-1$, then $f=0$. Therefore, $ev_{\mathtt D}$ is injective.  From Proposition \ref{pro10032026}, $\dim(\mathcal{L}(\mathtt{D},\mathtt{G}))=\textstyle{\frac{1}{2}}m(m-1)+1$.

To prove the lower bound on the minimum distance $d$, it is necessary to show that any
curve $\cV$ of equation (\ref{eq10032026A}) contains at most $q(m-1)-1$ points in $D$. Since $D$ is contained in $\cC_\tau$ but it possesses no point on the lines of equations $X_2=0$ and $X_3=0$, it is enough to show that if $g(X,Y)\in \mathbb{F}_{q^2}[X,Y]$ with $\deg(g(X,Y))\le m-2$, then the system of equations
$$\begin{cases}
Yg(X,Y)+\varepsilon X^m=0,\, \varepsilon \in \mathbb{F}_{q^2}, \\
\mbox{$Y-\tau X^{q+1}=0,\, \tau\ne 0$}
\end{cases}
$$
has at most $q(m-1)-1$ solutions $(\xi,\eta)$ with $\xi\ne 0$. Eliminating $Y$ gives equation  
$\tau X^{q+1}g(X,\tau X^{q+1})+\varepsilon X^m=0$. 
Its non-zero roots are also roots of the polynomial $p(X)=\tau X^{q+1-m} g(X,\tau X^{q+1})+\varepsilon$.   Since $\deg(g(X,Y))\le m-2$, the highest power of $X$ in $g(X,\tau X^{q+1})$ is at most $(q+1)(m-2)$. Therefore, $\deg(p(X))\le q+1-m+(q+1)(m-2)=q(m-1)-1$. 

To show the upper bound on the minimum distance $d$, it is sufficient to exhibit a curve $\cV$ of equation (\ref{eq09032026E}) passing through $(q+1)(m-2)$ pairwise distinct points in $D$. By Lemma \ref{lem09032026D} such a curve is $\mathcal{S}$ as defined in (\ref{eq10032026B}). 
\end{proof}
Magma supported computation for small values of $q$ suggests that the true minimum distance of $C_\mathcal{L}(\mathtt{D},\mathtt{G})$ may hit the upper bound in Theorem \ref{mainth}, namely $q^2-1-(q+1)(m-2)$. We show that this is the case when $m\le 3$.

%Passing to affine coordinates
%\[
%X=\frac{X_1}{X_3}, \qquad Y=\frac{X_2}{X_3},
%\]
%the points of  $D$ are of the form
%\[
%(X,Y)=(X,aX^{q+1}), \qquad X\in\mathbb{F}_{q^2}^*,
%\]
%where $a\in\mathbb{F}_{q^2}$ satisfies $a^q+a=1$.

%The  curve $\cV$ has affine equation
%\[
%f(XY)=Yg(X,Y)+\varepsilon X^{m}=0,\]
% with $\deg g\le m-2$.
%Therefore
%\[
%f(X,aX^{q+1})
%=aX^{q+1}g(X,aX^{q+1})+\varepsilon X^m.
%\]
%Since $X\neq 0$ on $D$, dividing by $X^m$ we obtain
%\[
%h(X):=X^{-m}f(X,aX^{q+1})
%=aX^{q+1-m}g(X,aX^{q+1})+\varepsilon.
%\]

%If $X^iY^j$ is a monomial of $g$, with $i+j\le m-2$, then after substituting
%$Y=aX^{q+1}$ it contributes to $h(X)$ a term of degree
%\[
%i+(q+1)j+q+1-m.
%\]
%Hence
%\[
%\deg h\le \max_{i+j\le m-2}\bigl(i+(q+1)j+q+1-m\bigr)=q(m-1)-1.
%\]

%Since $f\neq 0$, also $h\neq 0$, and so $h$ has at most $q(m-1)-1$ zeros.
%It follows that $f$ vanishes in at most $q(m-1)-1$ points of $D$.
%Therefore every non-zero codeword has weight at least
%\[
%(q^2-1)-(q(m-1)-1),
%\]
%and
%\[
%d\ge q^2-1-(q(m-1)-1).
%\]

%Finally, since
%\[
%\delta=(q^2-1)-\deg G=(q^2-1)-m(q-1),
%\]
%we get
%\[
%d-\delta\ge q+1-m,
%\]
%and the proof is complete.
\subsection{Case m=2}
\label{ss2}

\begin{theorem}
\label{th03042026B}    
Let $m=2$. Then  $C_\mathcal{L}(\mathtt{D},\mathtt{G})$ is a 
$[q^2-1, 2, q^2-q]_{q^2}$ cyclic two-weight code with non-zero weights $q^2-q$ and $q^2-1$. The codewords with weights $q^2-q$ are as many as $(q^2-1)(q+1)$, while those with weights $q^2-1$ are as many as $q(q-1)(q^2-1)$.  

%$
%q^2 - q$ and $ q^2 - 1$.
%Furthermore, the number of codewords of each weight is
%\[
%A_{q^2-q} = (q^2 - 1)(q + 1), \qquad
%A_{q^2-1} = q(q - 1)(q^2 - 1).
%\]
\end{theorem}
\begin{proof} In this case $q\ge 3$, and the code  $C_\mathcal{L}(\mathtt{D},\mathtt{G})$ has dimension $2$ over $\mathbb{F}_{q^2}$. % and contains $q^4$ codewords.
Every function $f \in \mathcal L( \mathtt{ G})$ is in the form 
$$
f=f(x,y) = a \frac{y}{x^2} + \varepsilon,
\qquad a,\varepsilon \in \mathbb{F}_{q^2}.
$$
If $a=0$, then $f$ has no zero on $D$ for $\varepsilon\in \mathbb{F}_{q^2}$, $\varepsilon \ne 0$, and this case occurs  $q^2-1$ times.

Assume that $a\neq 0$. Then the zeros of $f$ on $D$ are the roots in  $\mathbb{F}_{q^2}$ of the polynomial 
$$X^{q-1} =c,\,\mbox{with}\,\,  c=-\frac{\varepsilon}{a\tau}.$$
The polynomial $X^{q-1}-c$ with $c\in \mathbb{F}_{q^2}$, $c\neq 0$, has no root in $\mathbb{F}_{q^2}$ unless $c^{q+1}=1$, and in the exceptional case it has exactly $q+1$ roots. Thus, there are exactly $q^2-1-(q-1)$ values of $c$ for which the polynomial $X^{q-1}-c$ has no root in $\mathbb{F}_{q^2}$, and there are exactly $q+1$ values of $c$ for which it does. Thus the claim follows from the fact that each $c\in \mathbb{F}_{q^2}$, $c\neq 0$ is obtained exactly $q^2-1$ times when $a$ and $\varepsilon$ range over the non-zero elements in $\mathbb{F}_{q^2}$.   
%Since $C_\mathcal{L}(\mathtt{D},\mathtt{G})$  has as many as $q^4-1$ non zero codewords, the claim follows. 
\end{proof}
\subsection{Case m=3}
\label{ss3}
For $m=3$, the geometry of the functional code $C_\mathcal{L}(\mathtt{D},\mathtt{G})$ allows us, once again, to compute the true value of the minimum distance.  
\begin{proposition}
\label{pro13032026} For $m=3$,  the minimum distance $d$ in Theorem \ref{mainth} is equal to $q^2-q-2$, and hence it attains the upper bound.
\end{proposition}
\begin{proof} As in the proof of Theorem \ref{mainth}, it is necessary to show that any
curve $\cV$ of equation (\ref{eq10032026A}) contains at most $q+1$ points in $D$. Since $D$ is contained in $\cC_\tau$ but it possesses no point on the lines of equations $X_2=0$ and $X_3=0$, it is enough to show that if $g(X,Y)=b_0+b_1X+b_2Y\in \mathbb{F}_{q^2}[X,Y]$, then the system of equations
\begin{equation}
\label{eq04032026}
\begin{cases}
Y(b_0+b_1X+b_2Y)+\varepsilon X^3=0,\, \varepsilon \in \mathbb{F}_{q^2}, \\
\mbox{$Y-\tau X^{q+1}=0,\, \tau\ne 0$}
\end{cases}
\end{equation}
has at most $q+1$ solutions $(\xi,\eta)$ with $\xi\ne 0$. 

If $\varepsilon =0$, then $Y$ can be dismissed in the first equation. In other words, the first equation reads $b_0+b_1X+b_2Y=0$. Hence, the solutions of the system are given by the roots of the polynomial $b_0+b_1X+\tau b_2X^{q+1}$ and the claim follows for $\varepsilon=0$.    

Therefore, $\varepsilon=1$ may be assumed. We may also suppose that $\cV$ has no linear components of equation $Y-c$. 

Choose a non-zero element  $\eta \in \mathbb{F}_{q^2}$, from the value set of the polynomial $Y=\tau X^{q+1}$. Two cases arise according as the number of common points of $D$ with the line $\ell_\eta$ of equation $Y-\eta$ is at most one, or at least two. If the former case occurs for any non-zero $\eta \in \mathbb{F}_{q^2}$, then $\cV$ meets $D$ in at most $q-1$ points, and the claim is proven.   

Assume that the number of common points of $D$ and $\ell_\eta$  is at least two. Actually, that number cannot be two. In fact, if $Q=(\xi,\eta)\in \ell_\eta\cap \cV$ were not defined over $\mathbb{F}_{q^2}$, i.e. $\xi\not\in \mathbb{F}_{q^2}$ but $\eta\in \mathbb{F}_{q^2}$, then the Frobenius image $Q^{q^2}=(\xi^{q^2},\eta)$ of $Q$ would be a fourth common point of $\cV$ and $\ell_\eta$, whereas $\ell_\eta$ is not a component of $\cV$. Therefore, $\ell_\eta$ meets $D$ 
in three distinct points $Q_i=(\xi_i,\eta)$, $i=1,2,3$, lying in $PG(2,q^2)$. Then  $\xi_i\in \mathbb{F}_{q^2}$ is a root of the polynomial $h(X)=X^3-\eta b_1 X-\eta^2 b_2 -\eta b_0$. Since $X^2$ is missing in $h(X)$, the sum $\xi_1+\xi_2+\xi_3$ vanishes. Therefore, $\xi_1^q+\xi_2^q+\xi_3^q=0$. Since $\eta=\tau \xi_i^{q+1}$ for $i=1,2,3$, this yields that, up to the non-zero constant $\eta/\tau$,    
$$\frac{1}{\xi_1}+\frac{1}{\xi_2}=-\frac{1}{\xi_3}, $$
whence $\xi_3(\xi_1+\xi_2)=-\xi_1\xi_2$. Therefore, by $\xi_1+\xi_2=-\xi_3$, 
\begin{equation}
\label{eq15032026}
\xi_1\xi_2=\xi_3^2.
\end{equation}
Changing the roles of $\xi_3$ by $\xi_1$ and then also by $\xi_2$, (\ref{eq15032026}) reads 
\begin{equation}
\label{eq15032026A}
\mbox{$\xi_2\xi_3=\xi_1^2$ and $\xi_1\xi_3=\xi_2^2,$}
\end{equation}
respectively. From (\ref{eq15032026}) and (\ref{eq15032026A}), 
\begin{equation}
\label{eq15032026B}
\rho=\xi_1^3=\xi_2^3=\xi_3^3,
\end{equation}
and hence $(\xi_2/\xi_1)^3=1$. Since $(\xi_2/\xi_1)^{q+1}=1$,  this implies $q\equiv -1 \pmod{3}$. 
Furthermore, both $\xi_1$ and $\xi_2$ are roots of the linear polynomial 
$\eta b_1 X+\eta b_0+\eta^2 b_2-\rho$. Therefore, $b_1=0$. Thus, $\cV$ has equation 
\begin{equation}
\label{eq16032026G} 
Y(b_0+b_2Y)-X^3=0.
\end{equation}
The intersection of $\cV$ with $\cC_\tau$, whenever restricted to $D$, comprises the points $D_i=(\xi:\eta:1)$ whose coordinates $\xi$ are the solutions  in $\mathbb{F}_{q^2}$ of Equation 
\begin{equation}
\label{eq16032026F} b_2\tau^2 X^{2q-1}+b_0 \tau X^{q-2}=1.
\end{equation}
Since $\xi^{q^2}=\xi$,  raising to the $q$-th power shows that $\xi$ is also a solution in $\mathbb{F}_{q^2}$ of Equation
$$b_2^q\tau^{2q}\frac{1}{X^{q-2}}+b_0^q\tau^q\frac{1}{X^{2q-1}}=1.$$
Multiplying the above two Equations gives
\begin{equation}
\label{eq16032026H}
c_1 X^{q+1}+c_2\frac{1}{X^{q+1}}+c_3=0,
\end{equation}
with $c_1,c_2,c_3\in \mathbb{F}_q$ 
where 
$$c_1=\operatorname{N}(\tau)^2\operatorname{N}(b_2),\,\, c_2=\operatorname{N}(\tau)\operatorname{N}(b_0),\,\, c_3=\operatorname{Tr}(b_0^qb_2\tau^{q+2}-1)$$ and  $\operatorname{N}$ and $\operatorname{Tr}$ denote the norm and trace functions  from $\mathbb{F}_{q^2}$ to $\mathbb{F}_q$.

In particular, the solutions of (\ref{eq16032026H}) arise from those of the degree two equation $c_1Y^2+c_3Y+c_2=0$ over $\mathbb{F}_q$.  
Assume now that (\ref{eq16032026F}) has more than six solutions in $\mathbb{F}_{q^2}$.
By (\ref{eq16032026H}), there are at most two possible values of $X^{q+1}$, so that at least four solutions share the same value of $X^{q+1}$. Hence, for some solution $\xi_1$, there exist at least three further solutions of the form $\xi=\kappa \xi_1$ with $\kappa^{q+1}=1$.
Therefore, 
$$\begin{cases}
b_2\tau^2 \xi_1^{2q+2}+b_0 \tau \xi_1^{q+1}-\xi_1^3=0;\\ 
b_2\tau^2 \xi^{2q+2}+b_0 \tau \xi^{q+1}-\xi^3=b_2\tau^2(\kappa\xi_1)^{2q+2}+b_0 \tau(\kappa\xi_1)^{q+1}-\kappa^3\xi_1^3=0.
\end{cases}
$$
Since $\kappa^{q+1}=1$, this yields $\kappa^3=1$, and we have at most six solutions.  Therefore,
$\cV\cap \cC_\tau$ restricted to $D$ has size at most six. For $6\le q+1$, the claim follows. The remaining cases $q=3,4$ do not occur, as $q\equiv -1 \pmod{3}$.
\end{proof}
Our proof also shows that the codewords of minimum weight are obtained by evaluating non-zero functions in $\mathcal{L}(\mathtt{G})$ with $\varepsilon=0$.  The following claim provides a characterization of those functions. 
\begin{proposition}
\label{pro19032026} Let $m=3$ and $q>5$ and let $\tau$ be as in Lemma 2. A codeword $ev_{\mathtt{D}}(f)$ of $C_\mathcal{L}(\mathtt{D},\mathtt{G})$ has minimum weight 
if and only if 
$$f(x,y)=\frac{y(b_0+b_2y)}{x^3}$$
where $b_2\in \mathbb{F}_{q^2}^*$, $b_0/(\tau b_2)\in \mathbb{F}_q^*$. % and $b_0b_2\ne 0$.
The number of codewords with minimum weight is equal to $(q^2-1)(q-1)$. 
\end{proposition}
\begin{proof} From the proof of Proposition \ref{pro13032026}, if $q>5$, then a codeword $ev_{\mathtt{D}}(f)$ has minimum weight if and only if the associated polynomial $f(X)=\tau b_2X^{q+1}+b_1X+b_0$ with $b_0,b_1,b_2\in \mathbb{F}_{q^2}$ has $q+1$ (distinct) roots in $\mathbb{F}_{q^2}$. 

Lacunary polynomials of this kind appear in several applications of finite fields and they have been investigated in a series of papers, see for instance \cite{bl,kcm}. A main result in this direction is that any polynomial $u(X)=X^{q+1}+aX+b$, with coefficients in some finite field such that $q$ is a power of its characteristic, has either $0,1,2$, or $q+1$ roots over the same field. 

Now, assume that $f(X)$ has $q+1$ roots in $\mathbb{F}_{q^2}$. Then $\tau b_2\neq 0$. Replace $f(X)$ by $h(X)=X^{q+1}+c_1X+c_0$ where $h(X)$ arises from $f(X)$ dividing by $b_2\tau$. Then  $f(X)$ and $h(X)$ have the same roots. Take a root $\xi\in \mathbb{F}_{q^2}$ of $h(X)$. Then $\xi$ is also a root of the polynomial $g(X)=X^{q+1}+c_1^qX^q+c_0^q.$ So, 
$\xi$ is a root of the polynomial $k(X)=g(X)-h(X)=c_1^q X^q-c_1X+c_0^q-c_0$. If $f(X)$ and hence $k(X)$ has $q+1$ roots then $c_1=0$ and $c_0\in \mathbb{F}_q$. Therefore, $g(X)$ has $q+1$ distinct roots if and only if $c_0\in \mathbb{F}_q$ and $c_0\ne 0$. In terms of $f(X)$, we have $f(X)=\tau b_2 X^{q+1}+b_0$ where $b_0/( \tau b_2)\in \mathbb{F}_{q}$, $b_2\in \mathbb{F}_{q^2}$ and $b_0b_2\ne 0$. 
In particular, there exist as many as  $(q^2-1)(q-1)$ functions $f\in \mathcal{L}(\mathtt{D},\mathtt{G})$ with minimum weight.
\begin{remark} \emph{Proposition \ref{pro19032026} does not hold for $q=5$. In fact, a MAGMA aided computation shows for $q=5$ that the number of codewords of minimum distance of $C_\mathcal{L}(\mathtt{D},\mathtt{G})$ is equal to $672>96$. }
\end{remark}
\end{proof}\subsection{Roots of the polynomial $X^{q+1}(b_0+b_1X+b_2 X^{q+1})+X^3$}
%As in Lemma \ref{lem09032026B},
For $b_0,b_1,b_2,\tau \in \mathbb{F}_{q^2}$, $\tau\neq 0, b_2\neq 0$, we determine the non-zero roots of the polynomial 
$$g(X)=\tau X^{q+1}(b_0+b_1X+b_2 \tau X^{q+1})+X^3$$
using the field reduction technique. Replacing $\tau b_0,\tau b_1,\tau^2 b_2$ with $b_0,b_1,b_2$, respectively,
shows that $g(X)=X^{q+1}(b_0+b_1X+b_2 X^{q+1})+X^3$ may be assumed. 

We begin with the odd characteristic case. 
\subsection{Case $q\ge 5$ odd}
Fix a non-square element $s$ in $\mathbb{F}_q$. Look at $\mathbb{F}_{q^2}$ as the quadratic extension of $\mathbb{F}_q$ by the root $i$ of the irreducible equation $T^2-s=0$. Then $i^2=s$ and $i^q=-i$. Write $$b_0=\alpha_0+i \alpha_1,\, b_1=\beta_0+i\beta_1,\, b_2=\gamma_0+i\gamma_1.$$ Also, for a root $x\in \mathbb{F}_{q^2}$ of $g(X)$, write $x=x_1+i x_2.$ Then  $g(x)=A(x_1,x_2)+i B(x_1,x_2)=0$ where $A(x_1,x_2)=B(x_1,x_2)=0$ and 
$$
\begin{array}{llll}
A(x_1,x_2)=\big(\alpha_0+\beta_0 x_1+s\beta_1 x_2+\gamma_0(x_1^2-sx_2^2)\big)(x_1^2-sx_2^2)+x_1(x_1^2+3sx_2^2), \\
B(x_1,x_2)=\big(\alpha_1+\beta_1x_1+\beta_0x_2+\gamma_1(x_1^2-sx_2^2)\big)(x_1^2-sx_2^2)+x_2(3x_1^2+sx_2^2).
\end{array}
$$
In the projective plane $PG(2,q)$ over $\mathbb{F}_q$, let $\cF$ be the plane curve of affine equation $A(X_1,X_2)=0$ where 
$$ A(X_1,X_2)=\big(\alpha_0+\beta_0 X_1+s\beta_1 X_2+\gamma_0(X_1^2-sX_2^2)\big)(X_1^2-sX_2^2)+X_1(X_1^2+3sX_2^2).$$ 
The following results collect some useful properties of $\cF$. Let $R^+=(i:1:0)$ and $R^-=(-i:1:0)$. Also,
let $\ell^+$ and $\ell^-$ be the lines of equation $X_1=iX_2$, and $X_1=-iX_2$, respectively.  
\begin{lemma}
\label{lem27032026} Assume that $\alpha_0\ne 0$ and $\gamma_0\ne 0$.
\begin{itemize}
\item[(i)] $\deg(\cF)=4$.
\item[(ii)] The points at infinity of $\cF$ are $R^+$ and $R^-$.
\item[(iii)] The origin $O=(0:0:1)$ is a node with tangent lines $\ell^+$ and $\ell^-$.  %where $\ell^+: X_1=iX_2$ and %$\ell^-: X_1=-iX_2$. 
\item[(iv)] Both $R^+$ and $R^-$ are non-singular points. 
\item[(v)] If $\cF$ is reducible over an algebraic closure $\mathbb{K}$ of $\mathbb{F}_{q}$, then $\cF$ splits into two irreducible conics both defined over $\mathbb{F}_{q^2}$ but not over $\mathbb{F}_{q}$.
\end{itemize}
\end{lemma}
\begin{proof} The first three claims follow directly  from the equation of $\cF.$ 
To show (iv) assume on the contrary that $R^+$ is a singular point.  Then $I(R^+,\cF\cap\ell^+)\ge 2$. Also, $I(O,\cF\cap \ell^+)\ge 3$ as $\ell^+$ is a tangent line to $\cF$ at $O$. From the B\'ezout theorem, $\ell^+$ is a component of $\cF$. Since the Frobenius collineation preserves $\cF$, $R^-$ is also a singular point, and the previous argument shows that $\ell^-$ is also a component of $\cF$. Therefore, both $X_1-iX_2$ and $X_1+iX_2$, and hence their product $X_1^2-sX_2^2$ are factors of $A(X_1,X_2)$. This yields 
that $X_1^2-sX_2^2$ divides $X_1(X_1^2+3s X_2^2)$. But this is impossible $\mathbb{F}_q$ being supposed to have odd characteristic.  

To show (v) assume on the contrary that $\cF$ has a linear component $\ell$.  From (ii), $\ell$ is not the line at infinity $\ell_\infty$, and hence either $R^+\in\ell$, or $R^-\in \ell$. 
If $R^+\in\ell$ then (iv) yields that $\ell$ is the unique component through $R^+$. Therefore, $\ell$ is the tangent to $\cF$ at $R^+$. On the other hand, (ii) also shows that  $I(R^+,\cF\cap \ell_\infty)=2$, and hence $\ell_\infty$ is the tangent line to $\cF$ at $R^+$. But then $\ell=\ell_\infty$, a contradiction. A similar argument can be used when $R^-\in \ell$. Therefore, $\cF$ splits into two irreducible conics, say $C_1$ and $C_2$. From (ii), 
 $R^+\in \mathcal{C}_1$ or $R^-\in \mathcal{C}_1$. If $R^+\in \mathcal{C}_1$ then $R^+\notin \mathcal{C}_2$ by (i), and hence $I(R^+,\cF\cap \ell_\infty)=2$ implies that $\ell_\infty$ is the tangent line to $\mathcal{C}_1$ at $R^+$. Then $R^-\not\in \mathcal{C}_1$. Since the Frobenius collineation swaps $R^+$ and $R^-$, this implies that $\mathcal{C}_1$ is not defined over $\mathbb{F}_q$. 
 \end{proof}
 \begin{lemma}
\label{lem27032026D} Assume that $\alpha_0=0$ and $\gamma_0\ne 0$. Then (i),(ii), (iv) and  (v)
of Lemma \ref{lem27032026} hold. Moreover, 
\begin{itemize}
\item[(iii*)] The origin $O=(0:0:1)$ is a triple point.  
\end{itemize}
\end{lemma}
\begin{proof} 
%As before, let $\ell^+: X_1=iX_2$ and $\ell^-: X_1=-iX_2$ denote the lines joining $O$ with $R^+$, and $R^-$, respectively. 
%We show that no line through $O$ is a component. Let $\ell$ be a linear component through $O$ of equation $X_1=mX_2$ with $m\in \mathbb{F}$. Substituting $X_1$ by $m X_2$ in $A(X_1,X_2)$ gives $X_2^3\big((m\beta_0+s\beta_1)+m(m^2+s)+\gamma_0(m^2-s)X_2\big)$. If $\ell$ is a component, then both $m^2-s=0$ and $m\beta_0+s\beta_1+m(m^2+3s)=0$ hold whence $m=\pm i$, that is, 
%and $m(4s+\beta_0)+s\beta_1=0$. Since $\beta_0,\beta_1,s\in \mathbb{F}_q$, $s\ne 0$, this yields $4s+\beta_0=0$ and $\beta_1=0$.    
%that is $\ell$ is either $\ell^+$ or $\ell^-$, as defined in the proof of Lemma \ref{lem27032026}. Since the Frobenius collineation preserves $\cF$ and swaps $\ell^+$ and $\ell^-$, both of them are components of $\cF$. However, as in the proof of Lemma \ref{lem27032026}, this leads to a contradiction.  If $\ell$ has equation $X_2=mX_1$ with $m\in\mathbb{F}$, the same argument may be used. 
Using the same argument as in the proof of (v) of Lemma \ref{lem27032026}, it follows that  no line through $O$ is a component of $\cF$. Therefore, 
 if $\cF$ is reducible over $\mathbb{K}$, then its non-linear components through $O$ are irreducible conics over $\mathbb{K}$. Since $\deg(\cF)=4$, there are at most two such conics, but then $O$ is at most a double point, a contradiction with (iii*). This proves (v).     
\end{proof}
\begin{lemma}
\label{lem28032026B} Assume that $\alpha_0\neq 0$ and $\gamma_0=0$. Then (iii) of Lemma \ref{lem27032026} holds. Moreover, 
\begin{itemize}
\item[(i*)] $\deg(\cF)=3$.
\end{itemize}
\end{lemma}
\begin{lemma}
\label{lem29032026} Assume that $\alpha_0=\gamma_0=0$. Then (i*) and (iii*) in Lemma \ref{lem28032026B} hold. Moreover,
\begin{itemize}
\item[(vi)]  $\cF$ splits into three (not necessarily different) lines through $O$. 
\end{itemize}
\end{lemma}
\begin{proof} This time, $A(X_1,X_2)$ is a homogeneous polynomial of degree $3$, namely
$(\beta_0+1)X_1^3+s\beta_1X_1^2X_2+s(3-\beta_0)X_1X_2^2-s^2\beta_1X_2^3.$ 
Its factors are linear and define the three  lines in (vi).   
\end{proof}
 In $PG(2,q)$, let $\cG$ be the plane curve of affine equation $B(X_1,X_2)=0$ where
$$B(X_1,X_2)=\big(\alpha_1+\beta_1X_1+\beta_0X_2+\gamma_1(X_1^2-sX_2^2\big)(X_1^2-sX_2^2)+X_2(3X_1^2+sX_2^2).$$
The arguments used to prove Lemma \ref{lem27032026} also provide a proof for the following result. 
\begin{lemma}
\label{lem28032026A} The claims in Lemmas \ref{lem27032026}, \ref{lem28032026B}, \ref{lem27032026D} and \ref{lem29032026} hold true for $\cG$. 
\end{lemma}
\begin{lemma}
\label{lem28032026C} $\cF\ne \cG$. 
\end{lemma}
\begin{proof} The coefficients of $X_1^2X_2$ and $X_2^3$ in $A(X_1,X_2)$ are equal to  $s\beta_1$ and $-s^2\beta_1$, respectively. In $B(X_1,X_2)$, the homologous  coefficients are equal to $\beta_0+3$ and $-s(\beta_0-1)$, respectively. If $\cF=\cG$, then there exists a non-zero constant $c\in \mathbb{F}$ such that $c s\beta_1=\beta_0+3$ and $c s^2\beta_1=s(\beta_0-1)$. But this yields $\beta_0+3=\beta_0–1$ which may only happen when $\mathbb{F}_{q^2}$ has even characteristic.     \end{proof}
\begin{proposition}
    \label{pro28032026} If  $\gamma_0 \ne 0$ and $\gamma_1\ne 0$, then $\cF\cap \cG$ 
contains at most six affine points in $PG(2,q)$ other than $O$. 
\end{proposition}
\begin{proof} Assume that neither $\alpha_0\ne 0$, nor $\alpha_1\ne 0$. Then $O$ is a double point of both $\cF$ and $\cG$, and the lines $\ell^+$ and $\ell^-$ are common tangents to 
$\cF$ and $\cG$ at $O$. Therefore, $I(O,\cF\cap \cG)\ge 2\cdot 2+2\ge 6$. Moreover, 
$R^+$ and $R^-$ are common points of $\cF$ and $\cG$, and $\ell_\infty$ is their common tangent at $R^+$ and $R^-$. Therefore, $I(R^+,\cF\cap \cG)\ge 1+1=2$, and  $I(R^-,\cF\cap \cG)\ge 1+1=2$. Suppose that $\cF$ and $\cG$ do not have any common component over an algebraic closure $\mathbb{K}$ of $\mathbb{F}_q$. Then the B\'ezout theorem yields $|\cF\cap \cG|\le 16-6-4=6$, and the claim follows. Suppose that $\cF$ is reducible over $\mathbb{F}$. From (iv) of Lemma \ref{lem27032026},  $\cF$ splits into two irreducible conics $\mathcal{C}_1$ and $\mathcal{C}_2$ which are not defined over $\mathbb{F}_q$. Since both pass through $O$ and are tangent to $\ell_\infty$, each of them may have at most three affine points other than $O$. Therefore, the number of affine points of $\cF$ other than $O$ which are defined over $\mathbb{F}_q$ does not exceed $6$. From this the claim follows.  

Assume that $\alpha_0\ne 0$ but $\alpha_1=0$. Then $O$ is a triple point for $\cF$, and a double point for $\cG$. Therefore, $I(O,\cF\cap\cG)\ge 6$. As before, $I(R^+,\cF\cap \cG)\ge 1+1=2$, and  $I(R^-,\cF\cap \cG)\ge 1+1=2$. From the above argument, if $\cF$ is reducible over $\mathbb{K}$, then the claim holds as the number of affine points of $\cF$ other than $O$ which are defined over $\mathbb{F}_q$ does not exceed $6$. Therefore, $\cF$ may be assumed to be irreducible over $\mathbb{F}$. From the B\'ezout theorem, if $n$ is the number of common affine points of $\cF$ and $\cG$ distinct from $O$, we have $n\le 16-6-4=6$, and the claim follows. The same argument may be used for the case where $\alpha_0=0$ and $\alpha_1\neq 0$.

It remains to deal with the case where $\alpha_0=\alpha_1=0$. Then $O$ is a triple point for both $\cF$ and $\cG$. Therefore, $I(O,\cF\cap \cG)\ge 3\cdot 3=9$. This together with  $I(R^+,\cF\cap \cG)\ge 1+1=2$, and  $I(R^-,\cF\cap \cG)\ge 1+1=2$, yield that $|\cF\cap\cG|\le 16-9-4\le 3$ unless $\cF$ and $\cG$ have a common component. If this exception occurs, then (v) in Lemma \ref{lem27032026D} shows that $\cF$ splits into two irreducible conics defined over $\mathbb{F}_{q^2}$ but $\mathbb{F}_{q}$. Such a conic passes through $O$ and hence it has at most three other points in $PG(2,q)$. Therefore, $\cF$ may contain at most six points in $PG(2,q)$ other than $O$. 
\end{proof}
\begin{proposition}
\label{pro29032026B} If $\alpha_0=\gamma_0=0$, but either $\alpha_1\ne 0$, or $\gamma_1\ne 0$, then $\cF\cap \cG$ 
contains at most six affine points in $PG(2,q)$ other than $O$. 
\end{proposition}
\begin{proof} From (i) and (iii*) in Lemma \ref{lem29032026}, $\deg(\cF)=3$ and $O$ is a triple point of $\cF$. 

If $\gamma_1\neq 0$, then $\deg(\cG)=4$ and $O$ is either a node where $\cF$ and $\cG$ have the same two tangent lines, or a triple point of $\cG$, by Lemmas \ref{lem28032026A}, \ref{lem27032026} and \ref{lem27032026D}.   From the B\'ezout theorem, if $n$ is the number of common affine points of $\cF$ and $\cG$ distinct from $O$, we have $n\le 12-2\cdot 2\le 6$ and $n\le 12-2\cdot 3\le 6$, respectively, unless $\cF$ and $\cG$ have a common component. In the exceptional case, Lemma \ref{lem28032026A} together with  (v) in Lemma \ref{lem27032026D} show that $\cG$ splits into two irreducible conics defined over $\mathbb{F}_{q^2}$ but $\mathbb{F}_{q}$. As in the proof of Proposition \ref{pro28032026}, this implies that 
the number of affine points of $\cG$ other than $O$ which are defined over $\mathbb{F}_q$ does not exceed $6$.

If $\gamma_1=0$ and $\alpha_1\ne 0$, then $\deg(\cG)=3$ and $O$ is a node by Lemmas \ref{lem28032026A} and \ref{lem28032026B}. From the B\'ezout theorem, if $n$ is the number of common affine points of $\cF$ and $\cG$ distinct from $O$, we have $n\le 12-2\cdot 3\le 6$, unless $\cF$ and $\cG$ have a common component. As we have already pointed out in the previous argument, this yields that 
the number of affine points of $\cG$ other than $O$ which are defined over $\mathbb{F}_q$ does not exceed $6$.
\end{proof}
\begin{proposition}
\label{pro29032026C} Let $\alpha_0=\alpha_1=\gamma_0=\gamma_1=0$. Then the number of affine points in $\cF\cap \cG$ other than $O$, is either $0$ or $q-1$. 
\end{proposition}
\begin{proof} In this case, $b_0=b_2=0$, and $b_1\ne 0$ may be assumed. Then $g(X)=b_1\tau X^{q+1}+X^3$. Therefore, 
the non-zero roots of $g(X)$ are the roots of the polynomial $h(X)=b_1\tau X^{q-1}+1$. 
This polynomial $h(X)$ has either $q-1$ or no roots in $\mathbb{F}_q$ according as 
$(b_1\tau)^{q+1}=1$, or $(b_1\tau)^{q+1}\ne 1$.  It may be noticed that the former case occurs for as many as $q+1$ values of $b_1$.  
\end{proof}

We go on with the even characteristic case, and show that the above results established for the odd characteristic case hold true. We limit ourselves to writing explicitly down equations and conclusions whenever they differ from the odd characteristic case.   
\subsection{Case $q\ge 4$ even}
%We now deal with the case  $q$  even, and show that the same results as in the odd characteristic case hold.
 Let $\mathbb{F}_{q^2}=\mathbb{F}_q(\epsilon)$ where $\epsilon^2+\epsilon+\delta=0$ and $\operatorname{Tr}(\delta)=1$. Then $\epsilon^q=\epsilon+1$ and $\epsilon^{q+1}=\delta$.
%As in the previous section, we may assume $\tau=1$, and consider the equation
%\begin{equation} \label{eq1}
%	X^{q+1}(b_0 + b_1 X + b_2 X^{q+1}) + X^3 = 0.
%\end{equation}
Write
\[
b_0=\alpha_0+\epsilon\alpha_1,\,\,
b_1=\beta_0+\epsilon\beta_1,\,\,
b_2=\gamma_0+\epsilon\gamma_1,\,\, 
x=x_1+\epsilon x_2,
\qquad
\alpha_i,\beta_i,\gamma_i,x_i\in\mathbb{F}_q.
\]
Then $x^3=(x_1^3+\delta x_1x_2^2+\delta x_2^3)+
\epsilon(x_1^2x_2+x_1x_2^2+(\delta+1)x_2^3)$, and $x^{q+1}=x_1^2+x_1x_2+\delta x_2^2$. Therefore,
$g(x)=A(x_1,x_2)+\epsilon B(x_1,x_2)$
%$N=X^{q+1}$. Then
%$
%N=x_1^2+x_1x_2+\delta x_2^2
%$.
%Moreover, write
%$
%X^3=P+\epsilon Q,
%$
%where
%\[
%P=x_1^3+\delta x_1x_2^2+\delta x_2^3,\qquad
%Q=x_1^2x_2+x_1x_2^2+(\delta+1)x_2^3.
%\]
where 
$$
\begin{array}{lll}
A(x_1,x_2)=(x_1^2+x_1x_2+\delta x_2^2)\big(\alpha_0 + \beta_0 x_1 + \delta \beta_1 x_2 + \gamma_0 (x_1^2+x_1x_2+\delta x_2^2)\big)+\\
\qquad\qquad\quad\,\,   x_1^3+\delta x_1x_2^2+\delta x_2^3;\\
B(x_1,x_2)=(x_1^2+x_1x_2+\delta x_2^2)\big(\alpha_1 + \beta_1 x_1 + \beta_0 x_2 + \beta_1 x_2 + \gamma_1  (x_1^2+x_1x_2+\delta x_2^2)\big)+\\
\qquad\qquad\quad\,\,x_1^2x_2+x_1x_2^2+(\delta+1)x_2^3.
\end{array}
$$
In $PG(2,q)$, let $\cF$ be the plane curve of affine equation
$$
\begin{array}{lll}
A(X_1,X_2)=(X_1^2+X_1X_2+\delta X_2^2)\big(\alpha_0 + \beta_0 X_1 + \delta \beta_1 X_2 + \gamma_0 (X_1^2+X_1X_2+\delta X_2^2)\big)+\\
\qquad\qquad\quad\,\,   X_1^3+\delta X_1X_2^2+\delta X_2^3.\\
\end{array}
$$
Let $R^+=(\epsilon:1:0)$, and $R^-=(\epsilon^q:1:0)$, and 
%Separating the components with respect to the basis $\{1,\epsilon\}$ in \eqref{eq1}, we obtain the system
%\[
%\begin{cases}
%	F(x_1,x_2)=\gamma_0 N^2 + N(\beta_0 x_1 + \delta \beta_1 x_2) + P + \alpha_0 N = 0,\\[4pt]
%	G(x_1,x_2)=\gamma_1 N^2 + N(\beta_0 x_2 + \beta_1 x_1 + \beta_1 x_2) + Q + \alpha_1 N = 0.
%\end{cases}
%\]
%We now study the plane curves $F(X_1,X_2)=0$ and $G(X_1,X_2)=0$, starting from the following analogue of Lemma~4.
\begin{lemma}\label{lem4} Lemma \ref{lem27032026} holds true for $q$ even. 
\end{lemma}
\begin{proof}
	As $\gamma_0 \neq 0$, the leading term of $A(X_1,X_2)$ is $\gamma_0 (X_1^2+X_1X_2+\delta X_2^2)$, hence $\deg(\mathcal{F}) = 4$. The points of $\cF$ at infinity come from the zeros of the polynomial $X_1^2+X_1X_2+\delta X_2^2$ which are $(\epsilon,1)$ and $(\epsilon^q,1)$, up to a non-zero constant. Therefore, $R^+$ and $R^-$ are the points of $\cF$ at infinity. Moreover, from $\alpha_0 \neq 0$, the lowest degree term of $A(X_1,X_2)$ is $X_1^2+X_1X_2+\delta X_2^2$ which factors into two distinct polynomials, namely $X_1+\epsilon X_2$ and  $X_1 + \epsilon^q X_2$. 
    Thus $O$ is a node with tangent lines $\ell^+$ and $\ell^-$. Claims (iv) and (v) can be shown by arguing as in the odd characteristic case.
\end{proof}

\begin{lemma}\label{lemmaqev1} Lemma \ref{lem27032026D} holds true for $q$ even. 
\end{lemma}
\begin{proof}
%By Lemma~4, $\deg(\mathcal F)=4$, its points at infinity are $R_\epsilon$ and $R_{\epsilon^q}$, and both are non-singular.
This time, the lowest degree term in $A(X_1,X_2)$ is 
$$\beta_0(X_1^2+X_1X_2+\delta X_2^2)X_1+ X_1^3+\delta X_1X_2^2+\delta X_2^3$$
%Since $\alpha_0=0$, we have
%\[
%F(X_1,X_2)=\gamma_0 N(X_1,X_2)^2+N(X_1,X_2)(\beta_0X_1+\delta\beta_1X_2)+P(X_1,X_2),
%\]
%where
%\[
%N(X_1,X_2)=X_1^2+X_1X_2+\delta X_2^2,
%\qquad
%P(X_1,X_2)=X_1^3+\delta X_1X_2^2+\delta X_2^3.
%\]
%Hence the lowest-degree part of $F$ at the origin is
%\[
%(\beta_0X_1+\delta\beta_1X_2)+P(X_1,X_2),
%\]
which is a non-vanishing homogeneous polynomial of degree $3$. Therefore $O$ is a triple point. The other claims can be shown using the arguments from the proof of Lemma \ref{lem4}. 
%As in the proof of Lemma~4, $\mathcal F$ has no linear components. Therefore, if $\mathcal F$ were reducible over an algebraic closure of $\mathbb F_q$, then it would split into two irreducible conics. But this is impossible,
% since $O$ would be  at most a double point, a contradiction with (iii*). This proves (v) and in particular that $\mathcal F$ is absolutely irreducible.
\end{proof}
\begin{lemma}\label{lemmaqev2} Lemma \ref{lem28032026B} holds true for $q$ even.
%Assume that $\alpha_0 \neq 0$ and $\gamma_0 = 0$. Then (iii) of Lemma \ref{lem4} holds. Moreover,
%\item[(i$^\ast$)] $\deg(\mathcal F)=3$.
\end{lemma}
\begin{proof}
Since $\gamma_0=0$, the quartic term in $A(X_1,X_2)$ vanishes, and hence $\deg(\mathcal F)=3$. 
Since $\alpha_0\neq 0$, the lowest-degree term in $A(X_1,X_2)$ is
\[\alpha_0(X_1+\epsilon X_2)(X_1+\epsilon^qX_2).
\]
As $\epsilon\neq \epsilon^q$, the two tangent lines at $O$ are distinct. Therefore $O$ is a node.
\end{proof}

\begin{lemma} Lemma \ref{lem29032026} holds true for $q$ even.
%Assume that $\alpha_0=\gamma_0=0$. Then (i$^\ast$) and (iii$^\ast$) of Lemmas \ref{lemmaqev2} and \ref{lemmaqev1}  hold. Moreover,
%\item[(vi)] $\mathcal F$ splits into three (not necessarily distinct) lines through $O$.
\end{lemma}
\begin{proof}
In this case, 
\[
A(X_1,X_2)=(X_1^2+X_1X_2+\delta X_2^2)(\beta_0X_1+\delta\beta_1X_2)+ X_1^3+\delta X_1X_2^2+\delta X_2^3,
\]
which is a homogeneous polynomial of degree $3$. Therefore $O$ is a triple point and $\mathcal F$ splits into three (not necessarily distinct) lines through $O$.
\end{proof}
\begin{lemma}  
\label{lem02042026}
Lemma \ref{lem28032026A} holds true for $q$ even. 
%The same results proved above for the curve $\mathcal F$ hold true for the curve $\mathcal G$.
%\begin{lemma}
%	The plane curves $\mathcal F: F=0$ and $\mathcal G: G=0$ are distinct.
\end{lemma}
\begin{proof}
Assume on the contrary that $A(X_1,X_2)=c B(X_1,X_2)$ with a constant $c\in \mathbb{F}_q$. Comparing the coefficients of $X_1^3$, $X_1^2X_2$ and $X_1X_2^2$, we obtain 
%Since $F,G \in \mathbb{F}_q[X_1,X_2]$, we have $c \in \mathbb{F}_q$.
%				The coefficients of $X_1^3$, $X_1^2X_2$ and $X_1X_2^2$ in $F$ are
%		\[
%		\beta_0+1,\qquad \beta_0+\delta\beta_1,\qquad \delta(\beta_0+\beta_1+1),
%		\]
%		respectively, whereas the corresponding coefficients in $G$ are
%		\[
%		\beta_1,\qquad \beta_0+\beta_1+1,\qquad \beta_0+(1+\delta)\beta_1+1.
%		\]
		%Hence
	\begin{equation}\label{coef1}
			\beta_0+1=c\beta_1,
		\end{equation}
		\begin{equation}\label{coef2}
			\beta_0+\delta\beta_1=c(\beta_0+\beta_1+1),
		\end{equation}
		\begin{equation}\label{coef3}
			\delta(\beta_0+\beta_1+1)=c\bigl(\beta_0+(1+\delta)\beta_1+1\bigr).
		\end{equation}
				If $\beta_1=0$, then \eqref{coef1} gives $\beta_0=1$, and hence \eqref{coef2} yields $1=0$, a contradiction. Therefore,  $\beta_1\neq 0$.		
		From \eqref{coef1}, 
		$
		\beta_0+\beta_1+1=(c+1)\beta_1,
		$ and $
		\beta_0+(1+\delta)\beta_1+1=(c+1+\delta)\beta_1.
		$
		Substituting in \eqref{coef3} gives 
		$
		\delta(c+1)\beta_1=c(c+1+\delta)\beta_1
		$.
		Since $\beta_1\neq 0$, it follows 
		$
		\delta(c+1)=c(c+1+\delta)$,
		that is,
		$
		c^2+c+\delta=0$.		
		But this is a contradiction, as $\operatorname{Tr}(\delta)=1$ yields that the polynomial
		$T^2+T+\delta$ is irreducible over $\mathbb F_q$.		
		%Therefore $\mathcal F\neq \mathcal G$.
	\end{proof}

\begin{proposition} 
\label{pro02042026A}
All claims in Propositions \ref{pro28032026}, \ref{pro29032026B}, \ref{pro29032026C} hold true for $q\ge 4$. %Proposition \ref{pro29032026P} hold true for $q$ even. 
%If $\gamma_0\neq 0$ and $\gamma_1\neq 0$, then $\mathcal F\cap \mathcal G$ contains at most six affine points different from $O$.
\end{proposition}
\begin{proposition} 
\label{pro02042026B} Let $m=3$. If $q\ge 8$, then the third minimum weight is at least $q^2-7$. 
\end{proposition}
\subsection{Weight distribution}
Propositions 4,\ldots, 9  have the following corollaries.
\begin{theorem}
\label{th03042026} Let $m=3$, and $q\ge 4$. Then  $C_\mathcal{L}(\mathtt{D},\mathtt{G})$ is a 
$[q^2-1, 4, q^2-q-2]_{q^2}$ cyclic code with the following properties.  
\begin{itemize}
\item[(i)] The first minimum weight,  that is, the minimum distance is equal to $q^2-q-2$. The number of codewords with the first minimum weight is equal to $(q-1)(q^2-1)$.
\item[(ii)] For $q\ge 7$, the second minimum weight is equal to $q^2-q$. For $q\ge 8$, %a codeword $ev_{\mathtt{D}}(f)$ of $\mathcal{L}(\mathtt{G})$ has the second minimum weight 
%if and only if 
%$$f(x,y)=\frac{b_1xy+\varepsilon}{x^3}$$
%where $b_1^{q+1}=\tau^{q+1}$ and $\varepsilon\in \mathbb{F}_{q^2}$, $\varepsilon\ne 0$ 
the number of codewords with the second minimum weight is equal to $(q+1)(q^2-1)$.
\item[(iii)] For $q\ge 8$, the third minimum weight is at least $q^2-7$.  
\item[(iv)] The number of non-zero codewords with pairwise different weights is at most nine. 
%(The number of distinct non-zero weights is at most nine.) 
\end{itemize}
\end{theorem}
\begin{remark} \emph{The first claim in (ii) is not true for $q=5$, as a MAGMA aided computation shows for $q=5$ that the second minimum weight is equal to $19$.  
The second claim in (ii) does not hold for $q=7$. In fact, a MAGMA aided computation shows for $q=7$ that the number of codewords of minimum distance of $C_\mathcal{L}(\mathtt{D},\mathtt{G})$ is equal to $4992>384$. }
\end{remark}
\begin{remark}
\emph{It is plausible that equality should hold in (iii) and (iv), and this is confirmed by MAGMA aided computation for smaller values of $q$. However, a possible proof for the general case would require a much more detailed analysis of the set of common points of $\cF$ and $\cG$. }  
\end{remark}


\begin{thebibliography}{999}
 \bibitem{br} E. Ballico, A. Ravagnani, On the geometry of Hermitian one-point codes, \emph{J. Algebra}
{\bf{397}} (2014), 499-514.
\bibitem{bl} A.W. Bluher, On $x^{q+1}+ax+b$, \emph{Finite Fields Appl.}, {\bf{10}} (2004), 285-305.
 \bibitem{magma}W. Bosma, J. Cannon, C. Playoust, \textit{The Magma algebra system. I. The user language}, Journal of Symbolic Computation, 1997, 24, pp. 235-265.
% \bibitem{CT}C. Carvalho, F. Torres, On Goppa codes and Weierstrass gaps at several points, \emph{Des. Codes Cryptogr}, {\bf{35}} (2005), 211-225.
 \bibitem{CCPT} G. Caba$\rm{\tilde{n}}$a, M. Chara, R. Podest\'a, R. Toledano, On cyclic algebraic-geometry codes, \emph{Finite Fields Appl.} {\bf{82}} (2022), Paper No. 102064, 31 pp. 
\bibitem{har}  R.W. Hartley, Determination of the ternary collineation groups whose coefficients lie in the $GF(2^n$), \emph{Ann. of Math.} {\bf{27}} (1925/26), 140-158.
 \bibitem{Hirschfeld1}J. W. P. Hirschfeld, \textit{Projective Geometries over Finite Fields}, Oxford Mathematical Monographs, 1979.
 \bibitem{HKT}J. W. P. Hirschfeld, G. Korchm\'{a}ros, F. Torres, \textit{Algebraic curves over a finite field}, Princeton Series in Applied Mathematics. Princeton University Press, Princeton, NJ, 2008. xx+696 pp. 
% \bibitem{hirschfeld-storme-thas-voloch1991} J.W.P.~Hirschfeld, L.~Storme,
%J.A.~Thas, and J.F.~Voloch,  A characterization of Hermitian
%curves, \emph{J. Geom.} {\bf 41} (1991), 72--78.
  \bibitem{hoffer1972} A.R.~Hoffer,  On unitary collineation groups, \emph{J. Algebra} {\bf 22} (1972), 211--218.
 \bibitem{hughes-piper1973} D.R.~Hughes and F.C.~Piper, \emph{Projective Planes}, Graduate Texts in Mathematics {\bf 6}, Springer, New York, 1973, x+291 pp.
 \bibitem{kcm} K.H. Kim, J. Choe, S. Mesnager,  Solving $X^{q+1}+X+a=0$ over finite fields,
\emph{Finite Fields Appl.,} {\bf{70}} (2021), Article 101797. 
 \bibitem{oli} O.H.  King, The subgroup structure of finite classical groups in
 terms of geometric configurations Surveys in \emph{Combinatorics 2005}, pp. 29-- 56
 Cambridge University Press.
 \bibitem{KN1}G. Korchm\'{a}ros, G. P. Nagy, Hermitian codes from higher degree places, \emph{ J. Pure Appl. Algebra } {\bf{217}} (2013),  2371-2381.
 \bibitem{KN2}G. Korchm\'{a}ros, G. P. Nagy, Lower bounds on the minimum distance in Hermitian one-point differential codes, \emph{ Sci. China Math. } {\bf{56}} (2013), 1449-1455.
 \bibitem{KNT}G. Korchm\'{a}ros, G. P. Nagy, M. Timpanella, Codes and gap sequences of Hermitian curves, \emph{ IEEE Trans. Inform. Theory}, {\bf{66}} (2019), 3547-3554.
 \bibitem{KS}G. Korchm\'{a}ros, P. Speziali, Hermitian codes with automorphism group isomorphic to $PGL(2,q)$ with $q$ odd, \emph{Finite Fields Appl.} {\bf{44}} (2017), 1-17.
 \bibitem{M}G. L. Matthews, Weierstrass pairs and minimum distance of Goppa codes, \emph{Des. Codes Cryptogr.}, {\bf{22}} (2001), 107-121.
 \bibitem{mi} H.H. Mitchell, Determination of the ordinary and modular ternary linear
 groups, \emph{Trans. Amer. Math. Soc.} {\bf{12}} (1911), 207-242.
 \bibitem{P}O. Pretzel, \textit{Codes and algebraic curves}, Oxford Lecture Series in Mathematics and Its Applications, Clarendon Press, Oxford, 1998.
 \bibitem{Segre}B. Segre, Forme e geometrie hermitiane, con particolare riguardo al caso finito, \emph{Ann. Mat. Pura Appl.}, {\bf{70}} (1965), 1-201.
\bibitem{sti} H.~Stichtenoth, \emph{ Algebraic Function Fields and Codes},  Second edition. Graduate Texts in Mathematics, 254. Springer-Verlag, Berlin, 2009. xiv+355 pp.
\bibitem{XC} C.P. Xing and H. Chen, Improvements on parameters of one-point AG-codes from Hermitian codes, \emph{IEEE Trans. Inform. Theory} {\bf{48}} 2002, 535-537.
\bibitem{YK} K. Yang and P. V. Kumar, On the True Minimum Distance of Hermitian Codes, in \emph{Coding theory and algebraic geometry}, Lecture Notes in Mathematics, 1992, Volume {\bf{1518/1992}}, 99-10.
\end{thebibliography}
\end{document}